\documentclass[12pt,a4paper,oneside]{amsart}
\usepackage{color,amssymb,latexsym,amsfonts,textcomp,fancyhdr,calc,graphicx}
\usepackage{amsmath,amstext,amsthm,amssymb,amsxtra}
\usepackage[left=3cm,right=3cm,bottom=2cm,top=2cm]{geometry} %%%%%%%
\usepackage[colorlinks,citecolor=red,pagebackref,hypertexnames=false]{hyperref}
\usepackage{dsfont}
\usepackage[framemethod=tikz]{mdframed}
\usepackage[inline]{asymptote}

\usepackage[colorlinks,citecolor=red,pagebackref,hypertexnames=false]{hyperref}

\usepackage{txfonts,pxfonts,tikz} %also txfonts 
\usepackage{tgschola}
\usepackage[T1]{fontenc}

\numberwithin{equation}{section}

\theoremstyle{plain}
\newtheorem{theorem}{Theorem}[section]
\newtheorem{lemma}[theorem]{Lemma}

\theoremstyle{definition}

\newtheorem{case[theorem]}{Case}

\theoremstyle{remark}

\numberwithin{equation}{section}

\newcommand{\beql}[1]{\begin{equation}\label{#1}}
\newcommand{\eeq}{\end{equation}}
\newcommand{\comment}[1]{}

\newcommand{\Abs}[1]{{\left|{#1}\right|}}

\newcommand{\Set}[1]{{\left\{{#1}\right\}}}

\newcommand{\RR}{{\mathbb R}}

\newcommand{\ZZ}{{\mathbb Z}}

\newcommand{\QQ}{{\mathbb Q}}

\newcommand{\Span}{{\rm span\,}}

\begin{document}

\title{Deciding multiple tiling by polygons in polynomial time}

\author{Mihail N. Kolountzakis}
\address{Department of Mathematics and Applied Mathematics, University of Crete, Voutes Campus, 70013 Heraklion, Crete, Greece.}
\email{kolount@gmail.com}


\date{December 2019; revised May 2020}

\begin{abstract}
Suppose $P$ is a symmetric convex polygon in the plane. We give an algorithm,
running in time polynomial in the number of sides of the polygon, which decides
if $P$ can tile the plane by translations at some level (not necessarily at level one; this is multiple tiling). The main technical contribution is a polynomial time algorithm that selects, if this is possible, for each $j=1,2,\ldots,n$ one of two given vectors $e_j$ or $\tau_j$ so that the selection spans a discrete additive subgroup.
\end{abstract}

\maketitle

\tableofcontents

\section{Multiple tiling by translations with a single tile}

In this note we are discussing when a convex polygon can tile the plane by
translations with a lattice.
A lattice in the plane is a discrete additive subgroup which contains two linearly independent vectors.
Equivalently a lattice $L \subseteq \RR^2$ is a linear image of $\ZZ^2$, $L = A\ZZ^d$, where the matrix
$A$ is non-singular.

If $P$ is a convex polygon and $L$ is any subset of the plane
then we say that $P$ {\em tiles by translations} with $L$ if the translates
$$ P+\ell,\ \ (\ell \in L), $$ cover almost every point of the plane (in the sense of Lebesgue measure) exactly once. For
instance, if $P = [0, 1]^2$ is the unit square and $L = \ZZ^2$ then we do
indeed have tiling by translations. In this case the only points of the plane that are not
covered exactly once by the copies of $P$ are the points on the boundaries of these copies,
clearly a set of Lebesgue measure zero.

If we restrict ourselves to $P$ being a convex polygon it has long been known
that $P$ has to be symmetric in order to be able to tile the plane.  It is also
known that only parallelograms and symmetric hexagons admit tilings.

The problem becomes more interesting if one considers {\em multiple} tilings, i.e., tilings at higher level
than 1: almost every point of the plane is covered the same number of times
by the translates of $P$. This number is called the {\em level} of the tiling.
In this case many more convex polygons (again,
they must be symmetric) admit multiple tilings.  For instance, every symmetric convex
polygon whose vertices have rational coordinates admits multiple tiling by the lattice
$\ZZ^d$ at some level (this is easily seen from Theorem \ref{bolle} below).

It was Bolle \cite{Bo94} who first studied the question ``when does a polygon
tile multiply with a {\em given} lattice''. Bolle's characterization is the
following as was modified in \cite{K00}.

\begin{theorem} \label{bolle}
Let  $P$ be a convex polygon  in $\RR^2$, and $L$ be a lattice in $\RR^2$.
Then $P+L$ is a tiling of the plane at some level if and only if $P$ is centrally
symmetric, and for each pair of parallel edges $e$ and $e'$ of $P$ at least one
of the following conditions is satisfied:
\begin{itemize}
\item
{\bf (A1)} The translation vector, $\tau$, such that $e' = e + \tau$, is in $L$, or
\item
{\bf (B1)} $e\in L$ and {\bf (B2)} there exists $t\in(0, 1)$ such that $te+\tau\in L$.
\end{itemize}
\end{theorem}

It is, perhaps, worth noting that Theorem \ref{bolle} is true for more general polygonal domains.
As shown in \cite{K00} this characterization is true for any polygonal domain with the {\em pairing property}:
for every edge of the polygon there is exactly one other edge parallel to it and this edge has the same length
and opposite orientation. Every symmetric convex polygon clearly has this property but many more polygonal domains do,
even disconnected ones. See, for example, Fig.\ \ref{pairing-figure}.

\begin{figure}[ht] \label{pairing-figure}
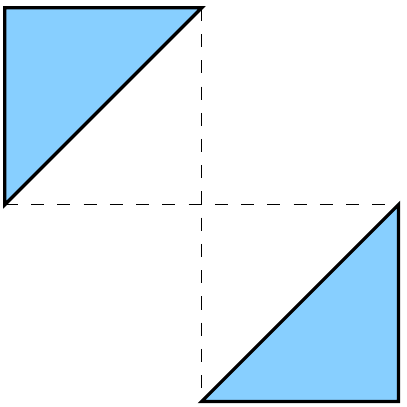
\caption{A polygonal domain with the pairing property. It is easy to see that it tiles with a lattice at level 1.}
\end{figure}

Theorem \ref{bolle} has been generalized in various ways
\cite{K00,GRS12,GKRS13,LL18,L18}. 

In \cite{L18}, in particular, B. Liu proves that every multiple translational tiling by a
convex polygon is necessarily periodic.  It is also proved in
\cite{L18} that if a convex polygon admits a multiple translational tiling then
it can also do so by lattice translations.  This means that if we care to
answer the question ``does this convex polygon tile multiply by translations'' then
the answer will be the same if we ask for lattice translations.

\begin{figure}[ht] \label{polygon-figure}
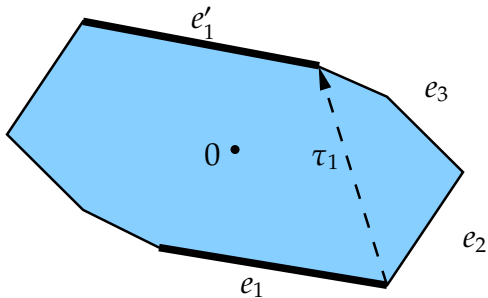
\caption{The edges $e_1, e_2, \ldots$ and the corresponding translation vectors $\tau_1, \tau_2, \ldots$ of a symmetric convex polygon.}
\end{figure}

For a symmetric convex polygon $P$ let us denote by
$$
e_1, e_2, \ldots, e_n, e_1', e_2', \ldots, e_n'
$$
its edges, in counterclockwise order.
By the symmetry of $P$ edge $e_1$ is parallel to edge $e_1'$, edge $e_2$ to edge $e_2'$ and so on.
Denote also by
$$
\tau_1, \tau_2, \ldots, \tau_n
$$
the translation vectors needed to carry $e_1$ to $e_1'$, $e_2$ to $e_2'$, etc.
In other words $e_j' = e_j + \tau_j$. See Fig.\ \ref{polygon-figure}.

It is mentioned in \cite{L18} that if one tries to directly use Theorem \ref{bolle} in deciding if a given polygon tiles multiply by {\em some} lattice,
one gets an exponential algorithm, due to the fact that one has to decide for each pair of opposite edges which of the two conditions in Theorem \ref{bolle} holds.
Indeed, as it is stated in \cite{L18}, one has to find a subset $J \subseteq [n] = \Set{1, 2, \ldots, n}$ such that
\beql{span1}
\Span_\ZZ\Set{\tau_j, e_{j'}:\ j \in J, j' \notin J}
\eeq
is a lattice, or show that no such subset exists.
(Here the set $J$ corresponds to the indices for which the alternative (A1) of Theorem \ref{bolle} holds; see \S \ref{sec:etau}.)
Clearly, if one tries all possible such $J$ one gets an exponential time algorithm.
The author then proceeds in \cite{L18} to give an alternative, polynomial time algorithm that decides if
a given symmetric convex polygon can tile multiply by translations with some lattice.

It is the purpose of this note to show a polynomial time algorithm for this question that is based directly on Theorem \ref{bolle},
showing, in effect, that the selection of $J \subseteq [n]$ such that \eqref{span1} is true can be done in polynomial time. The algorithm we present is not simpler or faster than that given in \cite{L18}. The main technical contribution of this note is the selection algorithm in \S \ref{sec:etau}, which is perhaps of independent interest.

%%%%%%%%%%%%%%%%%%%%%%%%%%%%%%%%%%%%%%%%%%%%%%%%%%%%%%%%%%%%%%%%%%%%%%%%%%%%%%%%%%%%%%%%%%%%%%%%%%%
%%%%%%%%%%%%%%%%%%%%%%%%%%%%%%%%%%%%%%%%%%%%%%%%%%%%%%%%%%%%%%%%%%%%%%%%%%%%%%%%%%%%%%%%%%%%%%%%%%%
%%%%%%%%%%%%%%%%%%%%%%%%%%%%%%%%%%%%%%%%%%%%%%%%%%%%%%%%%%%%%%%%%%%%%%%%%%%%%%%%%%%%%%%%%%%%%%%%%%%
\section{The algorithm}

%%%%%%%%%%%%%%%%%%%%%%%%%%%%%%%%%%%%%%%%%%%%%%%%%%%%%%%%%%%%%%%%%%%%%%%%%%%%%%%%%%%%
\subsection{The computational model.}\label{sec:model}
In view of the fact, mentioned in the previous section and easily derivable from Theorem \ref{bolle}, that
every symmetric polygon with rational vertices admits a multiple lattice tiling, we have to be careful to state exactly what kind of computational model we are assuming. It makes no sense, for instance, to assume that the coordinates of all our vertices are rational numbers (which would be the most natural thing, in terms of representation of a number in the computer) as then the answer would always be yes. So we have to allow for more general numbers in our model.

The simplest, though still unrealistic, model would be to assume that we can operate on arbitrary real numbers and that we can do arithmetic in constant time. It turns out that even this is not enough. It is important, as will be clear in the next section, to be able to tell when a number is rational.

For example, we will repeatedly have to decide if a given collection of vectors generates a discrete subgroup of $\RR^2$ or not.
This is easily done in polynomial time in our model.
Pick first one or two of these vectors which form a $\RR$-basis of the $\RR$-vector space that this collection generates.
For each other vector then check that its coefficients with respect to the chosen basis are rational.
All that is needed for the above is Gaussian elimination and in fixed dimension.

For simplicity we will carry out our analysis below using the model described above: arbitrary real numbers, constant time arithmetic operations, and we can tell in constant time if a number is rational. Then, we will explain how all our algorithms remain in polynomial time if we operate in the real-life computation model, of finite sized integers, where arithmetic operations take time that depends polynomially on the number of digits.

%%%%%%%%%%%%%%%%%%%%%%%%%%%%%%%%%%%%%%%%%%%%%%%%%%%%%%%%%%%%%%%%%%%%%%%%%%%%%%%%%%%%
\subsection{Selecting the $e$s and the $\tau$s that form a discrete group.}\label{sec:etau}
The next theorem is the main computational result of this note and tells us that the selection of the indices $j$ for which conditions (A1) or (B1) hold
can indeed be carried out in polynomial time. A little more work then, in \S\ref{sec:algorithm}, will manage to satisfy condition (B2) of Theorem \ref{bolle} as well.

\begin{theorem}\label{main}
Suppose $e_1,\ldots,e_n \in \RR^2$ are non-zero and pairwise non-parallel, and similarly for $\tau_1,\ldots,\tau_n \in \RR^2$.
We are seeking a set of indices
$$
J \subseteq [n] = \Set{1, 2, \ldots, n}
$$
such that the group
\beql{et}
\Span_\ZZ\Set{\tau_j, e_{j'}:\ j\in J,\ \ j' \in [n]\setminus J}
\eeq
is discrete.
We can find such a $J$ in polynomial time or decide that it does not exist.
\end{theorem}

\noindent{\em Proof.}
For $J \subseteq [n]$ write
$$
E(J) = \Span_\ZZ\Set{e_j:\ j \in J},\ \ \ T(J) = \Span_\ZZ\Set{\tau_j:\ j \in J}.
$$
With this notation we have that \eqref{et} is precisely the group
$$
T(J)+E([n] \setminus J).
$$

There is nothing to prove for $n\le 2$, since $\Span_{\ZZ} \Set{e_1, e_2}$ is always a lattice, so we are assuming from now on that $n \ge 3$.

First we check if $J=\emptyset$ satisfies the requirements of the Theorem. This is equivalent to $E([n])$ being a discrete group,
and we can check this in polynomial time as explained in \S \ref{sec:model}.

Second we observe that we can easily find such a $J$ of size 1, if it exists. Just examine each such $J$
and check if the set \eqref{et} is discrete. This clearly takes polynomial time.

In the analysis that follows we therefore search for sets $J$ with size at least 2.
Equivalently, since all $\tau_j$s are assumed non-parallel we are assuming that $T(J)$ contains two linearly independent vectors.

We now claim that if we can find $J \subseteq [n]$, with $\Abs{J} \ge 2$,  such that \eqref{et} is discrete not only
we have that $T(J)$ is discrete but we can also assume that $J$ is maximal with the property
\beql{property}
T(J) \text{ is discrete.}
\eeq
Let us prove this claim.
We may of course assume that $J$ is of {\em maximum size} such that \eqref{et} is discrete.
Then, if $J$ is not maximal with respect to property \eqref{property}, then there is $k \in [n]\setminus J$ such that $T(J \cup \Set{k}) \supseteq T(J)$ is discrete.
This implies that the group $T(J)$ has finite index in $T(J \cup \Set{k})$
(remember that we are assuming that $T(J)$ contains two linearly independent vectors).

We now state a useful lemma.
%%%%%%%%%%%%%%%%%%%%%%%%%%%%%%%%%%%%%%%%%%%%%%%%%%%%%%%%%%%%%%%%%%%%%%%%%%%%%%%%%%%%
\begin{lemma}\label{discrete}
Suppose $G_1\subseteq G$ and $H$ are subgroups of a topological abelian group.
Suppose the group $G_1$ has finite index in $G$ and $G_1+H$ is discrete.
Then $G+H$ is also discrete.
\end{lemma}

\begin{proof}
Since $[G:G_1] < \infty$ it follows that $G = G_1+F$ where $F$ is a finite set.
Then $G+H=G_1+H+F$ and this is discrete as it consists of finitely many translates of the discrete set $G_1+H$.

\end{proof}

Writing
$J' = [n] \setminus (J \cup \Set{k})$ we claim that
$$
T(J \cup \Set{k}) + E(J')
$$
is also discrete.
This follows from Lemma \ref{discrete} using
$$
G_1 = T(J),\ \ \ G = T(J\cup\Set{k}),\ \ \ H = E(J').
$$
This leads to contradiction with the fact that $J$ was assumed of maximum size and completes the proof of our claim
that it is enough to search for a $J$ which is maximal with respect to property \eqref{property}.

\begin{lemma}\label{maximal-intersection}
Suppose $A, B \subseteq [n]$ are both maximal with the property \eqref{property}.
Then $\Abs{A \cap B} \le 1$.
\end{lemma}

\begin{proof}
Suppose not, so that $i, k \in A\cap B$, $i\neq k$.
Write $\Gamma = \Span_\ZZ\Set{\tau_i, \tau_k}$.
Then $\Gamma$ is a lattice contained in $T(A) \cap T(B)$ and has finite index in
both $T(A)$ and in $T(B)$ (since these larger groups are also lattices).
We may of course assume that both $A$ and $B$ have at least three elements (otherwise one would contain
the other and one of them would not be maximal) and let then $j \in B \setminus A$.
Since $\Gamma \subseteq T(B)$ and $T(B)$ is discrete it follows that $\Gamma$ has finite index in
$$
\Gamma+\ZZ \tau_j = \Span_{\ZZ}\Set{\tau_i, \tau_k, \tau_j} \subseteq T(B).
$$
From Lemma \ref{discrete}, with $G_1=\Gamma$, $G=T(A)$, $H=\ZZ \tau_j$, we obtain that $T(A \cup \Set{j})$ is discrete,
which contradicts the maximality of $A$.

\end{proof}

By the observation above it is enough to search for a maximal set $J$ such that \eqref{et} is discrete.
By Lemma \ref{maximal-intersection} any two-element subset of such a maximal set $J$ determines $J$. Therefore the number of such $J$
is $O(n^2)$. We can enumerate them efficiently by starting with any set $J$ of size two and add elements to it that keep $T(J)$ discrete.
When we cannot do this any more we have found our unique maximal set that corresponds to (contains) the initial two-element set
that we started from. For each maximal such $J$ we only have to check that \eqref{et} is discrete,
which is doable in polynomial time.

The proof of Theorem \ref{main} is now complete.

%%%%%%%%%%%%%%%%%%%%%%%%%%%%%%%%%%%%%%%%%%%%%%%%%%%%%%%%%%%%%%%%%%%%%%%%%%%%%%%%%%%%
\subsection{Deciding tiling.}\label{sec:algorithm}

We divide the convex symmetric polygons $P$ that tile multiply into two classes according to the number of edge pairs for which condition (A1) of Theorem \ref{bolle} holds. Notice that since $\tau_1 = e_2 + e_3 + \cdots + e_n$ (if we orient the vectors counterclockwise -- refer to Fig.\ \ref{polygon-figure}) it is impossible for (A1) to hold for no index.

\begin{itemize}
\item {\bf Class A:} For all lattices $L$ for which $P+L$ is a multiple tiling we have that
condition (A1) of Theorem \ref{bolle} holds for exactly one index $j$.
\item {\bf Class B:} There exists a lattice $L$ for which $P+L$ is a multiple tiling and for which condition (A1) of Theorem \ref{bolle} holds for at least two indices.
\end{itemize}
To decide if $P$ tiles multiply we first check if it is in Class A. If that fails we check if it is in Class B.

\subsubsection{Deciding if $P$ is in Class A}\label{classA}

It is enough to test for all indices $j=1, 2, \ldots, n$, if (assume for simplicity $j=1$)
\begin{itemize}
\item[(I)] $L = \Span_\ZZ\Set{\tau_1, e_2, \ldots, e_n} = \Span_\ZZ\Set{e_2, \ldots, e_n}$ is a lattice,
\item[(II)] for $j = 2, \ldots, n$ the straight line $\tau_j + \RR e_j$ contains a point commensurable with $L$, and
\item[(III)] for $j = 2, \ldots, n$ the vector $\tau_j$ is incommensurable with $L$.
\end{itemize}
For a point $p$ to be commensurable with $L$ it is necessary and sufficient that $\ZZ p + L$ is a discrete group. 
We have already described in \S \ref{sec:model} how we can test this.

Suppose we tested condition (I) and it is true. Carry out a linear transformation so that $L = \ZZ^2$. We can now test if the straight line $\tau_j + \RR e_j$ contains a rational point by first writing our line in the form $y=ax+b$ (or $x = ay+b$) with $a \in \QQ$ (since (I) holds) and observing that this line contains a rational point if and only if $b \in \QQ$.

\subsubsection{Deciding if $P$ is in Class B}\label{classB}

Suppose $P$ is Class B and the lattice $L$ verifies this property.
Define $J \subseteq [n]$ to consist of those indices $j \in [n]$ for which $\tau_j \in L$.
We have $\Abs{J} \ge 2$.
By Theorem \ref{bolle} it follows that for $j' \in [n]\setminus J$ we have that $e_{j'} \in L$ and also condition (B2) holds for each such $e_{j'}$ for the lattice $L$.
Therefore $T(J) + E([n]\setminus J) \subseteq L$ and this group is a lattice.

Let $J' \supseteq J$ be maximal with the property that $T(J')$ is discrete.
It follows that for all $j' \in [n] \setminus J'$ we still have (B1) and (B2) holding for $L$. 
And we also have $T(J') \supseteq T(J)$, with $[T(J') : T(J)] < \infty$, since $T(J)$ is a lattice, $J$ containing at least two elements.
So
$$
L' := T(J') + E([n] \setminus J) \supseteq L
$$
is still a lattice by Lemma \ref{discrete}.
It follows that condition (A1) holds for
all $j \in J'$ with the lattice $L'$ in place of $L$, so that $T(J') \subseteq L'$.

Assume now that $j' \in [n]\setminus J'$.
Then $j' \notin J$ as well and, by Theorem \ref{bolle}, $e_{j'} \in L \subseteq L'$, so that
(B1) holds for $j'$ and the lattice $L'$.
Since $L' \supseteq L$ by Theorem \ref{bolle} again we obtain that (B2) holds for $j'$ and the lattice $L'$ as well.

By the above we conclude that, at the cost of enlarging the lattice $L$, we may assume that the set
$$
J = \Set{ j \in [n]:\ \tau_j \in L}
$$
is maximal with respect to the property that $T(J)$ is discrete. As reasoned in the proof of Theorem \ref{discrete} such a $J$ is determined by two of its points, so there is $O(n^2)$ possible such sets $J$. Our algorithm for determining then if $P$ is in Class B is the following.

\begin{itemize}
\item[1.]
For each set $J \subseteq [n]$ that is maximal with respect to the property that $T(J)$ is discrete:
\begin{itemize}
\item[2.] if $L = T(J)+E([n]\setminus J)$ is not discrete skip this set $J$
\item[3.] {\setlength{\itemindent}{10em} if it is discrete check that for each $j' \in [n]\setminus J$ the line $\tau_{j'} + \RR e_{j'}$ contains a point commensurable with $L$. If this happens then this set $J$ is good and $P$ is in Class B tiling with a superlattice of $L$.}
\end{itemize}
\end{itemize}
To effect the loop in 1.\ we proceed as in the proof of Theorem \ref{discrete}: for every pair of indices in $[n]$ we obtain the unique maximal $J$ containing the indices by augmenting them as long as the group they generate remains discrete. And the checks in 3.\ are done as described in \S\ref{classA}.

\subsection{A realistic computational model.}\label{sec:realmodel}

We now change to a more realistic computational model. One way this can be done is to assume that all our numbers will be in the {\em fixed} field $\QQ(a)$, where $a$ is an algebraic number of degree $d$ and known minimal polynomial $A(X) \in \QQ[X]$. This avoids the trivial field $\QQ$, where all polygons admit lattice tilings. Every one of our numbers can be represented in the form
$$
x = x_0 + x_1 a + x_2 a^2 + \cdots + x_{d-1} a^{d-1},\ \ \ x_j \in \QQ.
$$
Addition and subtraction of numbers work componentwise. Multiplication is multiplication of polynomials modulo $A(X)$ and to find $1/x$ we use
the extended Euclidean algorithm \cite[Chapter 3]{GG13}.
It is also clear, when $x$ is given in this representation,
if it is rational or not. All these operations take time polynomial in the size of the numbers involved, so our algorithm remains polynomial time in this model as well.

\bibliographystyle{amsplain}

\begin{thebibliography}{99}

\bibitem{Bo94} U. Bolle, On multiple tiles in $E^2$,
in {\em Intuitive Geometry}, K. B\"or\"oczky and G. Fejes T\'oth,
editors, Colloq.\ Math.\ Soc.\ J. Bolyai 63, 1994.

\bibitem{GG13} J. von zur Gathen and J. Gerhard,
Modern Computer Algebra, 3rd edition, 2013, Cambridge Univ.\ Press.

\bibitem{GKRS13}
N. Gravin, M.N. Kolountzakis, S. Robins, and D. Shiryaev.
Structure results for multiple tilings in 3{D}.
Discrete Comput. Geom., {\bf 50} (2013) 4, 1033--1050.
 
\bibitem{GRS12}
N. Gravin, S. Robins, and D.~Shiryaev.
Translational tilings by a polytope, with multiplicity.
Combinatorica, {\bf 32} (2012), 6, 629--649.

\bibitem{K00}
M.N. Kolountzakis,
On the structure of multiple translational tilings by polygonal regions,
Discr.\ Comp.\ Geom., {\bf 23} (2000), 4, 537-553.

\bibitem{LL18}                                                                                                                                         
N. Lev and B. Liu.
Multi-tiling and equidecomposability of polytopes by lattice translates,
Bull.\ London Math.\ Soc., {\bf 51} (2019), 6, 1079--1098.

\bibitem{L18}
B. Liu,
Periodic structure of translational multi-tilings in the plane,
arXiv preprint arXiv:1809.03440, 2018.


\end{thebibliography}

\end{document}